\documentclass[a4paper,10pt,reqno]{article}

\usepackage[utf8]{inputenc}
\usepackage[T1]{fontenc}
\usepackage{lmodern}
\usepackage[english]{babel}
\usepackage{microtype}

\usepackage{amsmath,amssymb,amsfonts,amsthm,esint}
\usepackage{mathtools,accents}
\usepackage{mathrsfs}
\usepackage{aliascnt}
\usepackage{braket}
\usepackage{bm}

\usepackage[a4paper,margin=3cm]{geometry}
\usepackage[citecolor=blue,colorlinks]{hyperref}

\usepackage{enumerate}
\usepackage{xcolor}

\makeatletter
\g@addto@macro\@floatboxreset\centering
\makeatother

\makeatletter
\def\newaliasedtheorem#1[#2]#3{
	\newaliascnt{#1@alt}{#2}
	\newtheorem{#1}[#1@alt]{#3}
	\expandafter\newcommand\csname #1@altname\endcsname{#3}
}
\makeatother

\numberwithin{equation}{section}

\newtheoremstyle{slanted}{\topsep}{\topsep}{\slshape}{}{\bfseries}{.}{.5em}{}

\theoremstyle{plain}
\newtheorem{theorem}{Theorem}[section]
\newaliasedtheorem{proposition}[theorem]{Proposition}
\newaliasedtheorem{lemma}[theorem]{Lemma}
\newaliasedtheorem{corollary}[theorem]{Corollary}
\newaliasedtheorem{counterexample}[theorem]{Counterexample}

\theoremstyle{definition}
\newaliasedtheorem{definition}[theorem]{Definition}
\newaliasedtheorem{question}[theorem]{Question}
\newaliasedtheorem{openquestion}[theorem]{Open Question}
\newaliasedtheorem{conjecture}[theorem]{Conjecture}

\theoremstyle{remark}
\newaliasedtheorem{remark}[theorem]{Remark}
\newaliasedtheorem{example}[theorem]{Example}

\newcommand{\setN}{\mathbb{N}}

\newcommand{\setR}{\mathbb{R}}

\newcommand{\eps}{\varepsilon}

\let\altphi\phi
\let\phi\varphi
\let\varphi\altphi
\let\altphi\undefined

\newcommand{\abs}[1]{\left\lvert#1\right\rvert}
\newcommand{\norm}[1]{\left\lVert#1\right\rVert}

\let\div\undefined
\DeclareMathOperator{\div}{div}

\newcommand{\di}{\mathop{}\!\mathrm{d}}

\newcommand{\loc}{{\rm loc}}

\DeclareMathOperator{\supp}{supp}

\newcommand{\leb}{\mathscr{L}}

\newfont{\tmpf}{cmsy10 scaled 2500}

\def\Xint#1{\mathchoice
	{\XXint\displaystyle\textstyle{#1}}%
	{\XXint\textstyle\scriptstyle{#1}}%
	{\XXint\scriptstyle\scriptscriptstyle{#1}}%
	{\XXint\scriptscriptstyle\scriptscriptstyle{#1}}%
	\!\int}
\def\XXint#1#2#3{{\setbox0=\hbox{$#1{#2#3}{\int}$ }
		\vcenter{\hbox{$#2#3$ }}\kern-.6\wd0}}

\def\dashint{\Xint-}

\begin{document}
	
	\title{Sharp regularity estimates for solutions of the continuity equation drifted by Sobolev vector fields}
	
	\author{Elia Bru\'e  \thanks{Scuola Normale Superiore, \url{elia.brue@sns.it, }}~~~~~ 
	Quoc-Hung Nguyen \thanks{Scuola Normale Superiore, \url{quochung.nguyen@sns.it, }} 
    \\\\ Scuola Normale Superiore,  Piazza dei Cavalieri 7, I-56100 Pisa, Italy. }
	\maketitle

	\begin{abstract}
		The aim of this note is to prove sharp regularity estimates for solutions of the continuity equation associated to vector fields of class $W^{1,p}$ with $p>1$.
		The regularity is of  ``logarithmic order'' and is measured by means of suitable versions of Gargliardo's seminorms.
	
	    \medskip
		
	\end{abstract}

\tableofcontents

\section{Introduction and main result}\label{sec: Introduction and main result}
In this paper we study regularity properties of solutions to the continuity equation in the $d$-dimensional Euclidean space: 
\begin{equation}\label{CE}
  \begin{dcases}
  \partial_t u+\div (bu)=0,\\
  u_0=\bar{u},
  \end{dcases}\tag{CE}
  \qquad \text{in}\ [0,T]\times\mathbb{R}^d,
\end{equation}
where $b: [0,T]\times\setR^d\rightarrow\setR^d$ is a time dependent vector field, $\bar{u}:\setR^d\rightarrow \setR$ is the initial data and $u:[0,T]\times \setR^d\to \setR$ is the unknown of the problem. 

We are mainly interested in the study of the Cauchy problem \eqref{CE} assuming the Sobolev regularity and the \textit{incompressibility condition} on the drift $b$, that is to say
\begin{equation*}
     \int_0^T \norm{b_s}_{W^{1,p}(\mathbb{R}^d)} \di s <\infty\quad
     \text{for some } p>1\text{ and } \div b_t(x)=0\quad \text{for a.e. }t \in [0,T],
\end{equation*}
where the divergence is understood in the distributional sense.
We will often assume a growth condition as $\norm{b}_{L^{\infty}([0,T]\times\setR^d)}<\infty$, look at \autoref{remark: growth conditions} for more details on this technical point. It is worth noticing that the just established setting is quite natural both from the theoretical point of view and for its applications to the study of nonlinear partial differential equations of the mathematical physics.

Solutions of \eqref{CE} are understood in the distributional sense in the class $L^{\infty}([0,T]\times \setR^d)$, more precisely we look at weak-star continuous maps $t \to u_t\in L^{\infty}(\setR^d)$ such that, for every $\phi\in C^{\infty}_c(\setR^d)$, the function $t\to \int_{\setR^d} \phi u_t \di x$ is absolutely continuous and fulfills
\begin{equation*}
	\frac{\di}{\di t}\int_{\setR^d} \phi u_t\di x=\int_{\setR^d} b_t\cdot \nabla \phi u_t \di x\qquad
	\text{for a.e.}\ t \in [0,T].
\end{equation*}
The Cauchy problem \eqref{CE} is strictly related to the system of ordinary differential equations
\begin{equation}\label{ODE}
\begin{dcases}
\frac{\di}{\di t} X(t,x)=b(t, X(t,x)),\\
X(0,x)=x, \quad (t,x)\in [0,T]\times\mathbb{R}^d.
\end{dcases}\tag{ODE}
\end{equation}
Indeed, when the vector field is regular enough (for instance globally bounded and Lipschitz in the spatial variable, uniformly in time) the classical Cauchy-Lipschitz theory grants the existence of a unique flow map $X:[0,T]\times \setR^d\rightarrow \setR^d$, furthermore the formula
\begin{equation}\label{z12}
	(X_t)_{\#} \bar{u} \leb^d=u_t\leb^d,
	\footnote{
		This is an identity between measures, where the left hand side is defined by $(X_t)_{\#} \bar{u} \leb^d (E):= (u\leb^d)((X_t)^{-1}(E))$ for every Borel set $E\subset \setR^d$. Note that \eqref{z12} is equivalent to 
		\begin{align*}
		\int_{\mathbb{R}^d}\varphi(x)u_t(x)dx=\int_{\mathbb{R}^d} \varphi (X(t,x))u_0(x)dx~~\forall~\varphi\in C_b(\mathbb{R}^d).
		\end{align*}
	}
\end{equation}
provides the unique solution $u_t$ of the Cauchy problem \eqref{CE}.

This link between \eqref{CE} and \eqref{ODE} is still present out of a smooth setting, but it is very subtle. The main technical issue is the loss of point-wise uniqueness for solutions of \eqref{ODE} when the drift is not Lipschitz. 
To overcome this difficult, Ambrosio in \cite{Ambrosio04} introduced  the notion of \textit{regular Lagrangian flow} (see \autoref{def:Regularlagrangianflow}) and established a link between well-posedness of the Cauchy problem \eqref{CE} and existence and uniqueness for regular Lagrangian flows. He also showed well-posedness in $L^{\infty}$ for \eqref{CE} when the vector field has the $BV$ spatial regularity and bounded divergence (it means $\div b_t\ll \leb^d$, with density in $L^{\infty}$); this result provides an important extension of the celebrated DiPerna-Lions theory \cite{lions}.
\paragraph*{}

In the last years the study of \textit{propagation of regularity} and \textit{bound of the rate of mixing} for solutions of \eqref{CE} has received a lot of attentions. At an informal level it consists in the study of the evolution of suitable norms, or functionals, along solutions of the continuity equation. For instance, in order to quantify the propagation of regularity one can study the evolution of Sobolev norms, $BV$ norms or weaker ones, while the ``mixing level'' of a scalar can be studied by means of negative Sobolev norms or geometrical functionals (see \cite{Bressan03, HSSS18}). 

We refer to \cite{AlbertiCrippaMazzuccato16,IyerKiselevXu14,HSSS18,Seis13} for an overview of the topic of mixing, and we are going to focus mostly on the regularity side of the problem.

In the smooth setting the picture is quite clear: the flow map $X_t$ and its inverse inherit the Lipschitz regularity of the vector field, precisely it holds
\begin{equation*}
	e^{-tL}|x-y|\le |X_t(x)-X_t(y)|\le e^{tL}|x-y|
	\qquad\forall x,y\in \setR^d,\ t\in [0,T],
\end{equation*}
where $L$ is the Lipschitz constant of $b$, uniform in time. This estimate can be immediately turned into
\begin{equation}\label{w1}
	\norm{\nabla u_t}_{L^{\infty}} \leq e^{tL} \norm{\nabla \bar u}_{L^{\infty}}\qquad \forall t\in [0,T],
\end{equation}
where $u_t$ is the unique solution of \eqref{CE} with initial data $\bar u\in W^{1,\infty}(\setR^d)\cap L^{\infty}(\setR^d)$, using the \textit{Lagrangian identity} \eqref{z12} and the incompressibility condition $\div b=0$.
Moreover it is a simple exercise to see that \eqref{w1} is sharp building a smooth divergence-free vector field $v$ and a solution of \eqref{CE}, associated to $v$, that increases the Lipschitz constant in time with exponential rate.

In other words the Lipschitz seminorm increases at most exponentially fast in time along solutions of the incompressible continuity equation drifted by Lipschitz velocity fields, and this rate is sharp.

If one consider vector fields of Sobolev class the situation is much more complicated and new wild phenomenas come up. 
The Lipschitz and Sobolev regularity, even of fractional order, of the initial data might be instantaneously lost during the time evolution (see \cite{AlbertiCrippaMazzuccato18, AlbertiCrippaMazzuccato16, AlbertiCrippaMazzuccato14}) but some very weak notion of regularity seems to be propagated also in this case.

A first positive result has been established by Crippa and De Lellis in  \cite{CrippaDeLellis08}, where they shown a quantitative Lusin-Lipschitz estimate for Lagrangian flows. This estimate allows to prove the propagation of the ``Lipexp'' regularity. In \cite{BreschJabin15} the authors proved that a suitable singular operator, evaluated on $u_t$, remains bounded during the time evolution, and by means of this regularity result they were able to built a new theory of existence of solutions to the compressible Navier-Stokes equation.
Using very sophisticated tools from harmonic analysis, in \cite{LegerFlavien16}, L\'eger studied the time evolution of
\begin{equation}\label{z13}
	\int_{\setR^d} |\log(|\xi|) \hat u_t(\xi)^2 \di \xi,
	\qquad
	\int_{\setR^d} \log(|\xi|)^2 \hat u_t(\xi)^2 \di \xi,
\end{equation}
where $\hat u_t$ denotes the Fourier transform of the solution to \eqref{CE} at time $t\in [0,T]$. He proved that, if  $b$ is divergence-free and belongs to $ W^{1,p}$, for some $p>1$ uniformly in time, then the first functional in \eqref{z13} increases at most linearly fast in time. Assuming a better regularity on the drift, i.e the $W^{1,p}$ bound with $p\geq 2$, the second functional increases at most quadratically fast.

\paragraph*{}

The main result in the present paper is a sharp characterization of the regularity for solutions of \eqref{CE} associated to incompressible Sobolev vector fields with exponent $p>1$. We study the propagation of regularity by means of the what we call \textit{log-Sobolev functionals} of order $p>0$:
\begin{equation}\label{eq: log-Sobolev norm}
\left(\int_{B_{1/3}} \int_{\setR^d} \frac{|f(x+h)-f(x)|^2}{|h|^d}\frac{1}{\log(1/|h|)^{1-p}} \di x \di h\right)^{1/2}.
\end{equation}
{\color{red}Inequality} \eqref{eq: log-Sobolev norm} is inspired by the well-known Gagliardo semi-norm
\begin{equation*}
\left(
   \int_{\setR^d} \int_{\setR^d} \frac{|f(x+h)-f(x)|^2}{|h|^{d+2s}} \di x \di h
\right)^{1/3}\qquad
s\in (0,1),
\end{equation*}
that aims at measure the ``size'' in $L^2$ of the $s$-derivative (i.e. derivative of order $s$) of $f$. At an intuitive level it is clear that replacing $|h|^{-2s}$ by $\log(1/|h|)^{1-p} \mathbf{1}_{B_{1/3}}(h)$ we are studying the ``$\log$-derivative'' of $f$, this justifies the name log-Sobolev functionals. This intuitive idea is also supported by the equivalence
\begin{equation}\label{interpolation-ine}
	\int_{B_{1/2}}\int_{\setR^d}\frac{|f(x+h)-f(x)|^2}{|h|^d}\frac{1}{\log(1/|h|)^{1-p}}\di x \di h  
	\simeq_{d,p}
	\int_{|\xi|\ge 10}\log(|\xi|)^p |\hat f(\xi)|^2\di \xi +\int_{|\xi|\leq 10}|\xi|^2|\hat f(\xi)|^2\di \xi,
\end{equation}
for every $f\in L^2(\setR^d)$, that is proven in \cite{BrueNguyen18}. In particular log-Sobolev functionals are comparable with the ones in \eqref{z13} when $p=1$ and $p=2$.
Our main result is the following.

\begin{theorem}\label{thm: intro}
	Let $p>1$ be fixed.
	Let us consider a bounded (in space and time) divergence-free vector field $b\in L^1([0,T]; W^{1,p}(\setR^d;\setR^d))$.
	
	Then for every initial data $u_0\in BV(\setR^d)$ with $\norm{\bar u}_{L^{\infty}}\leq 1$ the (unique) solution $u\in L^{\infty}([0,T]\times \setR^d)$ of \eqref{CE} satisfies
	\begin{equation}\label{z14}
	\int_{B_{1/3}} \int_{\setR^d} \frac{|u_t(x+h)-u_t(x)|^2}{|h|^d}\frac{1}{\log(1/|h|)^{1-p}} \di x \di h \lesssim_{p,d} \left(\int_0^t\norm{\nabla b_s}_{L^p} \di s\right)^p+  \norm{\bar{u}}_{BV}^p+\norm{\bar{u}}_{L^1}.
	\end{equation}

    Moreover, there exist a divergence-free vector field $b\in L^\infty([0,+\infty); W^{1,p}(\mathbb{R}^d))$ and an initial data $\bar{u}\in L^\infty(\setR^d)\cap W^{1,d}(\mathbb{R}^d)$, such that the unique solution $u\in L^\infty([0,+\infty)\times\mathbb{R}^d)$ of the Cauchy problem \eqref{CE} satisfies
	\begin{equation}
	\int_{B_{1/2}}\int_{\mathbb{R}^d}\frac{|u_t(x+h)-u_t(x)|^2}{|h|^d}\frac{1}{\log(1/|h|)^{\gamma}} \di x\di h=\infty,
	\qquad \forall t>0,
	\end{equation}
	for any $\gamma<1-p$.
\end{theorem}
If we further assume $b\in L^{\infty}([0,T]; W^{1,p}(\setR^d;\setR^d))$ with $p>1$, then \eqref{z14} ensures that the log-Sobolev functional \eqref{eq: log-Sobolev norm} of order $p$ increases in time at most polynomially fast with exponent $p$. Also this rate is sharp as is shown in \autoref{thm:sharpness2}.

We refer to \autoref{sec:reg} and \autoref{sec:counterexamples} for technical remarks on \autoref{thm: intro} concerning the boundedness assumption on the vector field, the regularity of the initial data and simple generalizations to the case of velocity fields with nonzero divergence.

Let us now spend a few words explaining the strategy of the proof of \autoref{thm: intro}.
For what concerns the first part of \autoref{thm: intro} our starting point has been the Lusin-Lipschitz estimate for Lagrangian flows obtained by Crippa and De Lellis in \cite{CrippaDeLellis08}. 
Indeed we show that a suitable version of this estimate (see \autoref{prop: regularity RLF}) implies \eqref{z14} by means of a general result (\autoref{prop: generalFact}) that has the aim to link a notion of ``having a logarithm Sobolev derivative'' written in terms of Lusin-Lipschitz property with a quantitative estimate in terms of log-Sobolev functionals.

The second part of \autoref{thm: intro} builds upon a variant of the construction proposed by Alberti Crippa and Mazzucato in \cite{AlbertiCrippaMazzuccato16} (see also \cite{AlbertiCrippaMazzuccato14} and \cite{AlbertiCrippaMazzuccato18}).
The main new technical tool we introduce is the interpolation inequality proved in \autoref{interpolation} (see also \autoref{cor: interpolation}) that links the log-Sobolev functionals \eqref{eq: log-Sobolev norm} of a function with its $L^2$ and $\dot H^{-1}$ norms.
As a byproduct of \autoref{thm: intro} and the just mentioned interpolation inequality we are able to recover the sharp bound on  ``mixing'' for vector fields with uniformly bounded $W^{1,p}$ norm, with $p>1$. This well-known result (see for instance \cite[Theorem 6.2]{CrippaDeLellis08}, \cite{IyerKiselevXu14}, \cite{HSSS18}, \cite{Seis13}, \cite{LegerFlavien16}) is proved in \autoref{prop: mixing estimate}.

\paragraph*{}

The remaining part of this paper is organized as follows. In \autoref{sec:reg} we deal with the first part of \autoref{thm: intro}, see also \autoref{thm:regularity result}. The second part of the paper, that is to say \autoref{sec:counterexamples}, is devoted to the proof of the second part of the \autoref{thm: intro} (see also \autoref{thm: sharpness}) and of \autoref{thm:sharpness2}. In this section we also collect two mixing estimates (see \autoref{prop: mixing estimate}) obtained as a byproduct of the previously developed theory.

\paragraph{Notation.}

We denote by $\setR^d$ the Euclidean space of dimension $d$ endowed with the Lebesgue measure $\leb^d$ and the Euclidean norm $|\cdot |$. $B_r(x)$ denotes the ball of radius $r>0$ centered at $x\in\setR^d$ and We often write $B_r$ instead of $B_r(0)$. $L^p=L^p(\setR^d)$ are standard Lebesgue spaces of $p$-integrable functions while $W^{1,p}(\setR^d)$ stands for the Sobolev space of functions endowed with the norm $\norm{f}_{W^{1,p}}^p=\norm{f}_{L^p}^p+\norm{\nabla f}_{L^p}^p$.

We set
\begin{equation*}
	\dashint_E f\di x=\frac{1}{\leb^d(E)} \int_E f \di x,\qquad \forall~ E\subset \setR^d\  \text{Borel set},
\end{equation*}
and
\begin{equation*}
	Mf(x):=\sup_{r>0} \dashint_{B_r(x)} |f(y)| \di y,
	\qquad
	\forall~x\in \mathbb{R}^d,
\end{equation*}
to denote the Hardy-Littlewood maximal function.

We often use the expression $a\lesssim_c b$ to mean that there exists a universal constant $C$ depending only on $c$ such that $a\leq C b$. The same convention is adopted for $\gtrsim_c$ and $\simeq_c$.


\section{Regularity result}\label{sec:reg}

In this section we deal with the positive part of \autoref{thm: intro} that is restated below for the reader convenience.

\begin{theorem}\label{thm:regularity result}
	Let $p>1$ be fixed.
	Let us consider a bounded (in space and time) divergence-free vector field $b\in L^1([0,T]; W^{1,p}(\setR^d;\setR^d))$.
	
	Then for every initial data $\bar{u}\in BV(\setR^d)$ with $\norm{\bar u}_{L^\infty}\le 1$ the (unique) solution $u\in L^{\infty}([0,T]\times \setR^d)$ of \eqref{CE} satisfies
	\begin{equation}\label{eq: LogSobolevContinuity}
	\int_{B_{1/3}} \int_{\setR^d} \frac{|u_t(x+h)-u_t(x)|^2}{|h|^d}\frac{1}{\log(1/|h|)^{1-p}} \di x \di h \lesssim_{p,d} \left(\int_0^t\norm{\nabla b_s}_{L^p} \di s\right)^p+ \norm{\bar{u}}_{BV}^p+\norm{\bar{u}}_{L^1}.
	\end{equation}
\end{theorem}
Let us begin with a few technical remarks.
\begin{remark}\label{remark: sup-norm}
	Thanks to \autoref{remark: sup-norm2}, we also have 
	\begin{equation*}
	 \sup_{h\in B_{1/3}} \log(1/ |h|)^{p}\int_{\mathbb{R}^d}|u_t(x+h)-u_t(x)|^2\di x
	 \lesssim_{p,d}
	 \left(\int_0^t\norm{\nabla b_s}_{L^p} \di s\right)^p+ \norm{\bar{u}}_{BV}^p+\norm{\bar{u}}_{L^1}.
	\end{equation*}
	This will play a role in the study of the geometric mixing norm along solutions of \eqref{CE}, see \autoref{prop: mixing estimate}.
\end{remark}

\begin{remark}\label{remark: growth conditions}
	The assumption $b\in L^{\infty}([0,T]\times \setR^d)$ can be replaced with more general growth conditions, for instance one can ask
	\begin{equation}\label{eq: moreGeneralGrowthCondition}
		\frac{b(t,x)}{1+|x|}=b_1(t,x)+b_2(t,x)
		\quad
		\text{with  $b_1\in L^1([0,T]; L^1(\setR^d))$ and $b_2\in L^1([0,T]; L^{\infty}(\setR^d))$},
	\end{equation}
	compare with \cite[pg. 12]{CrippaDeLellis08}.
	
	Let us point out that some growth condition on $b$ is necessary to ensure the existence of a unique regular Lagrangian flow, see \cite{Ambrosio04, AmbrosioCrippa14, AmbrosioColomboFigalli15, CrippaDeLellis08} for detailed discussions on this topic.
\end{remark}

\begin{remark}
	The divergence-free condition on $b$ can be weakened assuming
	\begin{equation}\label{eq: bounded divergence}
		\exp\left\lbrace \int_0^T \norm{\div b_s}_{L^{\infty}} \di s\right\rbrace	=L<\infty,
	\end{equation}
	provided we consider the \textit{transport equation}
	\begin{equation}\label{TrE}
		 \begin{dcases}
		\partial_t u+b\cdot \nabla u =0,\\
		u_0=\bar{u},
		\end{dcases}\tag{TrE}
	\end{equation}
	instead of \eqref{CE}. The rigorous statement is the following.
	
	Let $p>1$ be fixed.
	Let us consider a bounded vector field $b\in L^1([0,T]; W^{1,p}(\setR^d;\setR^d))$ satisfying \eqref{eq: bounded divergence}.
	Then for every initial data $\bar{u}\in BV(\setR^d)$ with $\norm{\bar u}_{L^{\infty}}\le 1$ the (unique) solution $u\in L^{\infty}([0,T]\times \setR^d)$ of the transport equation \eqref{TrE} satisfies
	\begin{equation*}
	\int_{B_{1/3}} \int_{\setR^d} \frac{|u_t(x+h)-u_t(x)|^2}{|h|^d}\frac{1}{\log(1/|h|)^{1-p}} \di x \di h \lesssim_{p,d} L^p\left(\int_0^t\norm{\nabla b_s}_{L^p} \di s\right)^p+ \norm{\bar{u}}_{BV}^p+\norm{\bar{u}}_{L^1}.
	\end{equation*}
\end{remark}

\begin{remark}
	The regularity assumption on the initial data $\bar{u}\in BV(\setR^d)$ is very far from being sharp, indeed it can be immediately weakened, for instance assuming either $\bar u\in W^{s,1}(\setR^d)$ for $0<s\le 1$ or that $\bar u$ satisfies a Lusin-Lipschitz regularity condition as \eqref{eq: Lusin-Lip exp}. 
	We expect that the minimal assumption to ask is
	\begin{equation}
		\int_{B_{1/3}} \int_{\setR^d} \frac{|u_0(x+h)-u_0(x)|^2}{|h|^d}\frac{1}{\log(1/|h|)^{1-p}} \di x \di h<\infty,
	\end{equation}
	but it seems difficult to prove this using our techniques. For sake of simplicity and for consistency with the counterexample in \autoref{thm: sharpness} we prefer to consider just Sobolev or $BV$ initial data.		
\end{remark}

The proof of \autoref{thm:regularity result}, as has been said in the introduction, builds upon a well-known ingredient: the quantitative Lusin-Lipschitz estimate for Lagrangian flows associated to Sobolev fields, first introduced in \cite{AmbrosioLecumberryManiglia} and \cite{CrippaDeLellis08}. In \autoref{sec: Regularity of Lagrangian flows} we state and prove the aforementioned regularity result for Lagrangian flows (see \autoref{prop: regularity RLF}), in a suitable formulation.  As a corollary we get a quantitative Lusin-Lipschitz estimate for solutions of the continuity equation (see \autoref{cor: LusinLipschitz continuity} and compare with \cite[Theorem 5.3]{CrippaDeLellis08}). Eventually in \autoref{sec: key lemma} we establish a general result that links a suitable quantitative Lusin-Lipschitz property of a scalar function with an estimate of the log-Sobolev functional \eqref{eq: log-Sobolev norm}.

\subsection{Regularity of Lagrangian flows}\label{sec: Regularity of Lagrangian flows}
 In this subsection we present a regularity estimate for Lagrangian flows associated to Sobolev vector fields with exponent $p>1$. 
 Let us begin by recalling the definition of Ambrosio's regular Lagrangian flows.

\begin{definition}\label{def:Regularlagrangianflow}
	Let us fix a time dependent vector field $b\in L^1_{\loc}([0,T]\times \setR^d;\setR^d)$. We say that $X:[0,T]\times \setR^d\rightarrow \setR^d$ is a regular Lagrangian flow associated to $b_t$ (RLF for short) if the following conditions hold:
	\begin{itemize}
		\item [(i)] there exists an $\leb^d$-negligible set $N\subset\setR^d$ such that
		\begin{equation*}
		X_t(x)=x+\int_0^t b_s(X_s(x))\di s \qquad \text{for any $t\in [0,T]$ and $x\in\setR^d\setminus N$};
		\end{equation*}
		\item [(ii)] there exists $L>0$, called compressibility constant, such that
		\begin{equation*}
		(X_t)_{\sharp} \leb^d \leq L\leb^d,\qquad\text{for any $t\in [0,T]$}.
		\end{equation*}
	\end{itemize}
\end{definition}
Condition (i) is not surprising, just ensures that $t\to X_t(x)$ solves \eqref{ODE} for $\leb^d$-a.e. initial data. Condition (ii), instead, has the role to select ``good'' trajectories imposing that the flow cannot concentrate too much the reference measure $\leb^d$. It also plays an important technical role guaranteeing that the notion of RLF is stable under modifications of the vector field on a $\leb^d$-negligible set. More precisely if $b$ and $\bar{b}$ satisfy 
\begin{equation*}
	\leb^{d+1}(\set{(t,x):\ |b(t,x)-\bar{b}(t,x)|>0})=0,
\end{equation*}
then $X$ is a RLF associated to $b$ if and only if it is a RLF associated to $\bar{b}$.

\begin{remark}\label{remark: existence uniqueness RLF}
	 It has been shown in \cite[Theorem 6.2, Theorem 6,4]{Ambrosio04} that, under the $BV$ assumption on the vector field $\int_0^T \norm{b_s}_{BV} \di s<\infty$, the uniform bound on the negative part of the divergence (i.e.  $\div b_t\ll \leb^d$ and $\int_0^T \norm{[\div b_s]^-}_{L^{\infty}}\di s<\infty$) and some growth conditions (for instance $b\in L^{\infty}([0,T]\times \setR^d)$ see also, \autoref{remark: growth conditions} ) there exists a unique RLF. See also the recent \cite{Nguyen18-1} for a different proof.
	 
	 It is also possible to see that, assuming a bound on the whole divergence
	 \begin{equation*}
	 \exp\left\lbrace  \int_0^T \norm{\div b_s}_{L^{\infty}}\di s	\right\rbrace \leq L,
	 \end{equation*}
	 then condition (ii) in \autoref{def:Regularlagrangianflow} can be improved
	 \begin{equation}\label{invertibility}
	 1/L \leb^d \leq (X_t)_{\sharp} \leb^d \leq L\leb^d,\qquad\text{for every $t\in [0,T]$}.
	 \end{equation}
	 In particular if $b$ is divergence-free then $X_t$ is measure preserving.
\end{remark}

\begin{remark}\label{remark: technical fact}
	For convenience we use the following convention: any vector field is \textit{point-wise} defined taking the value $\lim_{r\to 0}\frac{1}{\omega_d r^d}\int_{B_r(x)} b_t(y)\di y$ when $x\in\setR^d$ is a Lebesgue point for $b_t$ and $0$ otherwise.
	
	This convention allows for a point-wise \textit{Lusin-Lipschitz maximal estimate} for Sobolev vector fields:
	 \begin{equation}\label{eq: Lusin-Lip classic}
	 	|b_t(x)-b_t(y)|\lesssim_d |x-y|(M |\nabla b_t|(x)+M|\nabla b_t|(y))
	 	\qquad \forall x,y \in \setR^d\ \text{for}\ \leb^1\text{-a.e.}\ t\in [0,T],
	 \end{equation}
	 where $M$ is the Hardy-Littlewood maximal function.
	 See \cite{Stein} for a proof of this result at the level of scalar Sobolev functions.
\end{remark}

    We are ready to state and prove the Lusin-Lipschitz regularity result for Lagrangian flows associated to $W^{1,p}$ vector fields with $p>1$, compare it with \cite[Poposition 2.3]{CrippaDeLellis08}.
    
\begin{proposition}\label{prop: regularity RLF}
	Let $b\in L^1([0,T]; W^{1,p}(\setR^d;\setR^d))$ for some $p>1$, and $X$ be a regular Lagrangian flow associated to $b$ with compressibility constant $L$.
	Then, there exists a measurable function $g_t(x)=g(t,x):[0,T]\times\setR^d \to \setR \cup \set{+\infty}$ such that
	\begin{equation}\label{z8}
	e^{-g_t(x)-g_t(y)}
	\leq
	\frac{|X_t(x)-X_t(y)|}{|x-y|}\leq e^{ g_t(x)+g_t(y)}
	\qquad\text{for any $x,y\in\setR^d$ and $t\in [0,T]$}
	\end{equation}
	and 
	\begin{equation}\label{z2}
	\norm{g_t}_{L^p}\lesssim_{p,d} L\int_0^t \norm{\nabla b_s}_{L^p} \di s\qquad \forall t\in [0,T].
	\end{equation}
	Moreover, if $b$ is divergence-free then we can take $L=1$. 
\end{proposition}
\begin{proof}
	Let us take $N\subset\setR^d$ as in (i)\autoref{def:Regularlagrangianflow}, $\eps>0$, $x,y\in \setR^d\setminus N$ and $t\in [0,T]$, we have
	\begin{equation*}
	\abs{\log\left(\frac{\eps+|X_t(x)-X_t(y)| }{ \eps+|x-y| }\right)}
	=\abs{\int_0^t\frac{\di}{\di s} \log\left(\eps+|X_s(x)-X_s(y)|\right) \di s}
	\leq \int_0^t \frac{|b_s(X_s(x))-b_s(X_s(y))|}{|X_s(x)-X_s(y)|} \di s.
	\end{equation*}
	Using the Lusin maximal estimate \eqref{eq: Lusin-Lip classic} an letting $\eps\to 0$ we get
	\begin{equation*}
	\abs{\log\left(\frac{|X_t(x)-X_t(y)| }{|x-y| }\right)} \leq C_d \int_0^t M |\nabla b_s|(X_s(x))\di s+C_d\int_0^t M|\nabla b_s|(X_s(y)) \di s.
	\end{equation*}
	Set $g_t(x):=C_d\int_0^t M |\nabla b_s|(X_s(x))\di s$ when $x\in \setR^d\setminus N$ and $g_t(x):=+\infty$ otherwise. The condition (ii)\autoref{def:Regularlagrangianflow}, the boundness of the maximal function between $L^p$ spaces when $p>1$, together with Minkowski's inequality (see \cite[Appendix]{Stein}), yields \eqref{z2}. The proof is complete.
\end{proof}

\begin{remark}
	It is clear from the proof of \autoref{prop: regularity RLF} that $g_t$ can be taken independent of $t$, simply considering $g_T:= C_d\int_0^T  M |\nabla b_s|(X_s(x))\di s$.
\end{remark}
As a simple consequence of \autoref{prop: regularity RLF} we get a Lusin-Lipschitz estimate for solutions of the continuity equation \eqref{CE}, compare with \cite[Theorem 5.3]{CrippaDeLellis08}.

\begin{corollary}\label{cor: LusinLipschitz continuity}
	Let $b\in L^1([0,T]; W^{1,p}(\setR^d;\setR^d))$ be bounded and divergence-free with $p>1$.	
    Then, there exists a measurable function $\tilde g_t(x)=\tilde g(t,x):[0,T]\times \setR^d\to \setR\cup \set{+\infty}$, such that, for every initial data $\bar{u}\in BV(\setR^d)\cap L^{\infty}(\setR^d)$,  there exists a representative $u:[0,T]\times \setR^d\to \setR^d$ of the (unique) solution in $L^{\infty}([0,T]\times\setR^d)$ to \eqref{CE} satisfying
        \begin{equation}\label{z3'a}
        |u_t(x)-u_t(y)|\leq |x-y|\exp\left\lbrace \tilde{g}_t(x)+\tilde{g}_t(y)\right\rbrace,
        \qquad
        \forall x,y\in \setR^d,\ \ \forall t\in [0,T],
        \end{equation}
    and
        \begin{equation}\label{z3'}
        \norm{\tilde{g}_t}_{L^p}\lesssim_{p,d} \int_0^t \norm{\nabla b_s}_{L^p} \di s+\norm{\bar{u}}_{BV}\qquad \forall t\in [0,T].
        \end{equation}	
\end{corollary}

\begin{proof}
	Let $X$ be a unique regular Lagrangian flow associated to $b$. For any $t\in [0,T]$ the map $x\to X_t(x)$ is essentially invertible, i.e. there exists $Y:[0,T]\times \setR^d \to \setR^d$ such that
	\begin{equation}\label{z15}
		X(t,Y(t,x))=Y(t,X(t,x))=x \qquad \text{for}\ \leb^d\text{-a.e.}\ x\in \setR^d,
	\end{equation}
	see \cite[Theorem 6.2]{Ambrosio04}.
	It is immediate to check that $Y_t$ satisfies the inequality \eqref{z8} replacing $g_t(x)$ with $\bar g_t (x):=g_t(Y_t(x))$ when $x$ fulfills \eqref{z15} and $\bar g_t(x):=\infty$ otherwise. Using that $Y_t$ preserves $\leb^d$ (compare with \autoref{remark: existence uniqueness RLF}) it is immediate to verify that
	\begin{equation}\label{z3}
		\norm{\bar g_t}_{L^p}\lesssim_{p,d} \int_0^t \norm{\nabla b_s}_{L^p} \di s\qquad \forall t\in [0,T].
	\end{equation}
	Thus, up to modify again $\bar g_t$ on a negligible set, we get
	\begin{equation*}
	\frac{|u_t(x)-u_t(y)|}{|x-y|}\le C_d (M|\nabla \bar{u}|(Y_t(x))+M|\nabla \bar{u}|(Y_t(y)))\exp\left\lbrace\bar g_t(x)+\bar  g_t(y)\right\rbrace
	\qquad
	\forall x,y\in \setR^d,\ \  \forall t\in [0,T],
	\end{equation*}
	 where we used the $\leb^d$-a.e. identity $u_t=\bar u(Y_t)$ (it can be checked observing that $u_t(X_t(x))=\bar u (x)$ for $\leb^d$-a.e. $x\in \setR^d$) and \eqref{eq: Lusin-Lip classic} for $\bar u \in BV(\setR^d)$. 
	 
	 Finally observe that, for $x,y\in \setR^d$ and $t\in [0,T]$, one has
	\begin{equation*}
	C_d(M|\nabla \bar{u}|(Y_t(x))+M|\nabla \bar{u}|(Y_t(y))) \exp\left\lbrace \bar g_t(x)+ \bar g_t(y)\right\rbrace\leq \exp\left\{\tilde{g}_t(x)+\tilde{g}_t(x)\right\},
	\end{equation*}
	where $\tilde{g}_t(x)=\bar 2g_t(x)+2\log\left( \max\set{C_dM|\nabla\bar u|(Y_t(x));1} \right)$. This implies \eqref{z3'a}. Thanks to the $L^p$ bound of the maximal function (see \cite[Theorem 1]{Stein}), the fact that $Y_t$ is measure preserving and \eqref{z3} we obtain \eqref{z3'}. The proof is complete.
\end{proof}

We conclude this subsection by proving a converse of \autoref{prop: regularity RLF}. Roughly speaking we show that, for a given vector field $b\in L^1([0,T]\times \setR^d;\setR^d)$, the existence of a RLF that satisfies a version of the Lusin-Lipschitz estimate \eqref{z8} with exponent $p>1$ implies that $b$ is of class $W^{1,p}$. 
We are grateful to Luigi Ambrosio for having pointed this out to us.

Let us recall that a vector field $b\in L^1(\setR^d)$ belongs to $\text{BD}(\setR^d)$ if its distributional \textit{symmetric derivative} $Eb$ is a finite Radon measure, we refer to \cite{AmbrosioCosciaDalMaso97} for more explanations.

\begin{theorem}\label{th: converse}
	Let $b\in L^1([0,T]\times\setR^d; \setR^d)$ be fixed. If there exists a Regular Lagrangian flow $X_t$ associated to $b$ that satisfies
	\begin{align}\label{eq:important}
		\exp\left\lbrace -\int_s^t h_r(X_r(x))\di r-\int_s^t h_r(X_r(y))\di r\right\rbrace
		&\le \frac{|X_t(x)-X_t(y)|}{|X_s(x)-X_s(y)|}\\
		&\le \exp\left\lbrace \int_s^t h_r(X_r(x))\di r+\int_s^t h_r(X_r(y))\di r\right\rbrace
		\nonumber
	\end{align}
	for any $x,y\in\setR^d$ and $0\le s<t\le T$, where $h:[0,T]\times\setR^d\to [0,+\infty]$ fulfills
	\begin{equation}
		\int_0^T\norm{h_t}_{L^p}\di t<\infty,
		\qquad\text{for some $p\ge 1$,}
	\end{equation}
	then
	\begin{itemize}
		\item[(i)] when $p=1$, $b\in L^1([0,T];\text{BD}(\setR^d))$, and $\int_0^T\norm{b_t}_{BD}\le \int_0^T\norm{b_t}_{L^1}\di t+2\int_0^T\norm{h_t}_{L^1} \di t$;
		\item[(ii)] when $p>1$, $b$ admits a distributional derivative in $L^p$ for $\leb^d$-a.e. $t\in [0,T]$ 
		\begin{equation*}
			\int_0^T \norm{\nabla b_t}_{L^p}\di t\lesssim_{d,p} \int_0^T\norm{h_t}_{L^p}\di t.
		\end{equation*}
	\end{itemize}
\end{theorem}

\begin{remark}
	It is clear from the proof of \autoref{prop: regularity RLF} that, if $b\in L^1([0,T];W^{1,p})$ for some $p>1$ and $\div b\in L^{\infty}([0,T]\times\setR^d)$ then \eqref{eq:important} holds true with $h_r=M|\nabla b_r|\in L^p(\setR^d)$.
\end{remark}

\begin{proof}[Proof of \autoref{th: converse}]
	Let $N\subset\setR^d$ as in \autoref{def:Regularlagrangianflow}, we have
	\begin{align*}
		\int_s^t h_r(X_r(x))\di r+\int_s^t h_r(X_r(y))\di r  \ge & 
		\abs{\log\left( \frac{|X_t(x)-X_t(y)|}{|X_s(x)-X_s(y)|}\right) }
		=  \abs{\int_s^t \frac{\di}{\di r}|X_r(x)-X_r(y)|\di r}\\
		= & \abs{\int_s^t \frac{\big < b_r(X_r(x))-b_r(X_r(y));\  X_r(x)-X_r(y)\big >}{|X_r(x)-X_r(y)|^2}\di r}
	\end{align*}
	for any $x,y\in \setR^d\setminus N$. Therefore, for any $x,y\in \setR^d\setminus N$, it holds
	\begin{equation*}
		\abs{ \frac{\big < b_r(X_r(x))-b_r(X_r(y));\  X_r(x)-X_r(y)\big >}{|X_r(x)-X_r(y)|^2}}
		\le h_r(X_r(x))+h_r(X_r(y))
		\quad\text{for $\leb^1$-a.e. }r\in [0,T],
	\end{equation*}
	thus 
	\begin{equation}
	\abs{ \frac{\big < b_r(x)-b_r(y);\  x-y\big >}{|x-y|^2}}
	\le  h_r(x)+ h_r(y)
	\quad\text{for $\leb^1$-a.e. }r\in [0,T],\ \forall x,y\in \setR^d\setminus N_r,
	\end{equation}
	where $\leb^d(N_r)=0$.
	Our conclusion follows from \autoref{lemma:Sobolev} below.
\end{proof}

\begin{lemma}\label{lemma:Sobolev}
	Let $b\in L^1(\setR^d;\setR^d)$ and $h\in L^p(\setR^d)$ for $p>1$. If
	\begin{equation*}
		\abs{ \frac{\big < b(x)-b(y);\  x-y\big >}{|x-y|^2}}
		\le h(x)+ h(y)
		\quad \text{for any $x,y\in \setR^d\setminus N$, where $\leb^d(N)=0$},
	\end{equation*}
	then 
	\begin{itemize}
		\item[(i)] when $p=1$, $b\in \text{BD}(\setR^d)$ and $\norm{b}_{\text{BD}}\le \norm{b}_{L^1}+2\norm{h}_{L^1}$;
		\item[(ii)] when $p>1$, $b$ admits a distributional derivative in $L^p$ and $\norm{\nabla b}_{L^p}\lesssim_{d,p} \norm{h}_{L^p}$.
	\end{itemize}
\end{lemma}
\begin{proof}
	The proof is based on a simple approximation argument.
	Let $\rho_{\eps}(x):=\eps^{-d}\rho(x/\eps)$ where $\rho\in C^{\infty}_c(\setR^d)$ is nonnegative, supported in $B_1$ with $\int \rho =1$. Set $b_{\eps}:=\rho_{\eps}\ast b$ and $h_{\eps}:=\rho_{\eps}\ast h$, it is easily seen that
	\begin{equation}\label{q1}
		\abs{ \frac{\big < b_{\eps}(x)-b_{\eps}(y);\  x-y\big >}{|x-y|^2}}
		\le h_{\eps}(x)+ h_{\eps}(y),
		\quad\forall x,y\in \setR^d.
	\end{equation}
	Since $b_{\eps}$ and $h_{\eps}$ are smooth functions, using the Taylor expansion of $b^{\eps}$ around $x\in\setR^d$, from \eqref{q1}, we deduce $|\nabla_{\text{sym}} b_{\eps}(x)|\le 2 h_{\eps}(x)$ where we denoted by $\nabla_{\text{sym}}$ the symmetric part of the gradient.
	It gives
	\begin{equation*}
	\sup_{\eps\in (0,1)} \norm{\nabla_{\text{sym}}b_{\eps}}_{L^p}\le 2\norm{h}_{L^p},
	\end{equation*}
	that implies (i). When $p>1$ Korn's inequality (see \cite[Remark 7.11]{AmbrosioCosciaDalMaso97}) implies
	\begin{equation*}
		\sup_{\eps\in (0,1)} \norm{\nabla b_{\eps}}_{L^p}\lesssim_{p,d}
		\sup_{\eps\in (0,1)} \norm{\nabla_{\text{sym}} b_{\eps}}_{L^p}\le
		 2\norm{h}_{L^p}
	\end{equation*}
	and thus (ii) follows.
\end{proof}

\subsection{The key lemma}\label{sec: key lemma}
This section is devoted to the proof of the following.

\begin{proposition}\label{prop: generalFact}
      Let $p\ge 1$ be fixed. For any $f\in L^1(\setR^d)$ satisfying the exponential Lusin-Lipschitz estimate
      \begin{equation}\label{eq: Lusin-Lip exp}
        |f(x)-f(y)|\leq |x-y|\exp\left\lbrace g(x)+g(y) \right\rbrace\quad \forall x,y\in \setR^d,
        \text{ for some } g\in L^p(\setR^d),
     \end{equation}
     it holds
	 \begin{equation}\label{eq: LogSobolev}
	   \int_{B_{1/3}} \int_{\setR^d} \frac{1\wedge|f(x+h)-f(x)|^2}{|h|^d}\frac{1}{\log(1/|h|)^{1-p}} \di x \di h \lesssim_{p,d} \norm{g}_{L^p}^p+ \norm{f}_{L^1}.
	\end{equation}
\end{proposition}

Roughly speaking this result establishes an implication between two different notions of ``having a derivative of logarithmic order''. Note that these notions are not equivalent, indeed every H\"older function $f$ satisfies $\int_{B_{1/3}} \int_{\setR^d} \frac{1\wedge|f(x+h)-f(x)|^2}{|h|^d}\frac{1}{\log(1/|h|)^{1-p}} \di x \di h<\infty$, but there are H\"older functions fulfilling $\liminf_{y\to x}\frac{|f(x)-f(y)|}{|x-y|}=\infty$ for any $x\in \setR^d$, therefore they do not satisfy \eqref{eq: Lusin-Lip exp}.

\begin{remark}\label{remark: moreGeneral fact}
	The result in \autoref{prop: generalFact} is written in a form useful for our purposes and can be generalized in many ways. For instance one can assume that $f\in L^1(\setR^d)$ satisfies a H\"older-Lipschitz inequality
	 \begin{equation*}
    	|f(x)-f(y)|\leq |x-y|^{\alpha}\exp\left\lbrace g(x)+g(y) \right\rbrace\qquad \forall x,y\in \setR^d,
	\end{equation*}
	for some $\alpha\in (0,1]$ and some $g\in L^p(\setR^d)$ and prove
	\begin{equation*}
	\int_{B_{1/3}} \int_{\setR^d} \frac{1\wedge|f(x+h)-f(x)|^2}{|h|^d}\frac{1}{\log(1/|h|)^{1-p}} \di x \di h \lesssim_{p,\alpha,d} \norm{g}_{L^p}^p+ \norm{f}_{L^1}.
	\end{equation*}
    The proof follows verbatim the one in \autoref{prop: generalFact}.
\end{remark}
Let begin by proving a technical lemma.

\begin{lemma}\label{lemma: elementary computation}
   Let $f\in L^1(\setR^d)$ and $g\in L^p(\setR^d)$ be as in \autoref{prop: generalFact}.
   Then it holds
   \begin{equation}
   	\int_{\setR^d} 1\wedge |f(x+h)-f(x)|^2\di x \lesssim_d |h|^2  \int_1^{\log(1/|h|)} e^{2\lambda} \leb^d(\set{2g>\lambda}) \di \lambda+ |h|\norm{f}_{L^1} ,
   \end{equation}
   for every $h\in \setR^d$ with $|h|\leq 1/e$.   
\end{lemma}

\begin{proof}
Using \eqref{eq: Lusin-Lip exp} and Cavalieri's summation formula we get
\begin{align*}
	\int_{\setR^d}  1\wedge  |f(x+h)-&f(x)|^2  \di x\\
     =&2\int_{0}^{e|h|} t \leb^d\left(\set{x:\ |f(x+h)-f(x)|>t}\right)\di t\\
     &+2\int_{e|h|}^1 t\leb^d\left(\set{x:\ |f(x+h)-f(x)|>t }    \right)\di t\\
    \lesssim & |h|\norm{f}_{L^1}+ \int_{e|h|}^1 t \leb^d\left(\set{x:\ |f(x+h)-f(x)|>t }    \right)\di t
     \\
	\lesssim & |h|\norm{f}_{L^1}+\int_{e|h|}^{1}t \leb^d(\set{x:\ g(x)+g(x+h)>\log(t/|h|)}) \di t,
\end{align*}
for $\leb^d$-a.e. $h\in \setR^d$ with $|h|\leq 1/e$.
Estimating $\int_{e|h|}^{1}t \leb^d(\set{x:\ g(x)+g(x+h)>\log(t/|h|)}) \di t$ with
\begin{equation}\label{z1}
2\int_{e|h|}^{1}t \leb^d(\set{2g>\log(t/|h|)}) \di t,
\end{equation}
setting $\lambda=\log(t/|h|)$ and changing variables in \eqref{z1} we conclude the proof.
\end{proof}
	\begin{remark}\label{remark: sup-norm2}
	Observe that \autoref{lemma: elementary computation} immediately gives a weak version of \autoref{prop: generalFact}:
	\begin{equation*}
	\sup_{h\in B_{1/3}}  \log(1/|h|)^p	\int_{\setR^d} 1\wedge |f(x+h)-f(x)|^2\di x\lesssim_{p,d}  \norm{g}_{L^{p,\infty}}^p+ \norm{f}_{L^1},
	\end{equation*}
	where $L^{p,\infty}$ is the weak $L^p$ space. 
	Indeed we have
	\begin{align*}
	|h|^2  \int_1^{\log(1/|h|)} e^{2\lambda} \leb^d(\set{2g>\lambda}) \di \lambda\lesssim |h|^2 \norm{g}_{L^{p,\infty}}^p \int_1^{\log(1/|h|)} e^{2\lambda} \lambda^p \di \lambda \lesssim_p \log(1/|h|)^p  \norm{g}_{L^{p,\infty}}^p.
	\end{align*}
	\end{remark}

\begin{proof}[Proof of \autoref{prop: generalFact}]
	In order to shorten notation we set $\mu(\lambda):=	\leb^d( \set{2g>\lambda}) \di \lambda$. Using the result in \autoref{lemma: elementary computation} we get
\begin{align*}
		\int_{B_{1/e}} \int_{\setR^d} &\frac{1\wedge|f(x+h)-f(x)|^2}{|h|^d\log(1/|h|)^{1-p}}  \di x \di h\\
		 &\lesssim  \int_{B_{1/e}} \frac{\log(1/|h|)^{p-1}}{|h|^{d}}\left( |h|^2\int_1^{\log(1/|h|)} e^{2\lambda}\di \mu(\lambda)+ |h|\norm{f}_{L^1} \right) \di h\\
		 &\lesssim_{p,d} \int_0^1 \log(1/r)^{p-1}r\int_1^{\log(1/r)} e^{2\lambda}\di \mu(\lambda) \di r+\norm{f}_{L^1}.
\end{align*}
Changing variables according to $\log(1/r)=t$ and applying Fubini theorem we get
\begin{align*}
\int_0^1 \log(1/r)^{p-1}r\int_1^{\log(1/r)}& e^{2\lambda}\di \mu(\lambda) \di r\\
= & \int_0^{\infty} e^{-2t} t^{p-1} \int_1^t e^{2\lambda} \di \mu(\lambda) \di t\\
= & \int_1^{\infty} e^{2\lambda} \int _{\lambda}^{\infty} e^{-2t}t^{p-1} \di t\di \mu(\lambda).	
\end{align*}	
Using the integration by part formula and the inequality $\lambda\ge 1$ it is elementary to check that 
\begin{equation}
	e^{2\lambda} \int _{\lambda}^{\infty} e^{-2t}t^{p-1} \di t
	\lesssim_p \lambda^{p-1},
\end{equation}	
that together with the definition of $\mu(\lambda)$ implies
\begin{equation*}
   \int_1^{\infty} e^{2\lambda} \int _{\lambda}^{\infty} e^{-2t}t^{p-1} \di t\di \mu(\lambda)\lesssim_p \int_0^{\infty} \lambda^{p-1} \di \mu(\lambda)
   \lesssim_p \norm{g}_{L^p}^p.
\end{equation*}	
Putting all together we get
\begin{equation}
\int_{B_{1/e}} \int_{\setR^d} \frac{1\wedge|f(x+h)-f(x)|^2}{|h|^d}\frac{1}{\log(1/|h|)^{1-p}} \di x \di h \lesssim_{p,d} \norm{g}_{L^p}^p+ \norm{f}_{L^1},
\end{equation}
that clearly implies our conclusion
\end{proof}
	
Eventually the proof of \autoref{thm:regularity result} follows applying \autoref{prop: generalFact} with $f=u_t$, and recalling \autoref{cor: LusinLipschitz continuity} and \autoref{remark: norm is conserved} below.

\begin{remark}\label{remark: norm is conserved} 
	If $b\in L^1([0,T];BV(\setR^d;\setR^d))$ with bounded divergence then any solution $u_t$ of \eqref{CE} in $L^{\infty}$ has \textit{the renormalization property} (see \cite{Ambrosio04}) that is to say, for any $\beta\in C_c^1(\mathbb{R})$ the function $\beta(u(t,x))$ is a solution of the continuity equation as well. It implies that, for every $1\le p<\infty$, the $L^p$ norm of $u_t$ is preserved in time.
	
	Note that in the smooth setting, at least for $p=2$,
	 this property comes from a simple computation:
	\begin{align*}
	\frac{\di}{\di t}\frac{1}{2}\int_{\mathbb{R}^d}|u_t|^2\di x
	=-\int_{\mathbb{R}^d}u_t\div(b_t u_t) \di x
	=-\frac{1}{2} \int_{\setR^d} \div(b_tu_t^2)\di x=0.
	\end{align*}
	
\end{remark}

\section{Counterexamples and mixing estimates}\label{sec:counterexamples}
The aim of this section is to show the sharpness of \autoref{thm:regularity result} in the scale of log-Sobolev functionals. 

The first result ensures that the polynomial growth of order $p$ proved in \eqref{eq: LogSobolevContinuity} is sharp. In the case $p=1$ the result is already known and has been established by L\'eger in \cite{LegerFlavien16}.

\begin{theorem}\label{thm:sharpness2}
	Let $p\geq 1$. There exist a smooth divergence-free vector field $b$ belonging to $L^{\infty}([0,+\infty); W^{1,\infty}(\setR^d,\setR^d))$ supported in $B_1\times
	[0,+\infty)$, and a smooth initial data $\bar u$ supported in $B_1$, such that the unique solution $u\in L^{\infty}([0,+\infty)\times\mathbb{R}^d)$ of the continuity equation \eqref{CE} satisfies
	\begin{equation*}
	\int_{B_{1/3}}\int_{\mathbb{R}^d}\frac{|u_t(x+h)-u_t(x)|^2}{|h|^d}\frac{1}{\log(1/|h|)^{1-p}} \di x\di h\gtrsim  t^p,
	\end{equation*}
	for any $t\in [0,+\infty)$.
\end{theorem}

The second and most important example has been already presented in the introduction (part two of \autoref{thm: intro}) and is restated below for the reader convenience.

\begin{theorem}\label{thm: sharpness}
Let $p\geq 1$. There exist a divergence-free vector field $b\in L^\infty((0,+\infty); W^{1,p}(\mathbb{R}^d))$ supported in $B_1\times
[0,+\infty)$, and an initial data $\bar u\in L^\infty(\setR^d)\cap W^{1,d}(\mathbb{R}^d)$ also supported in $B_1$, such that the unique solution $u\in L^\infty([0,+\infty)\times\setR^d)$ of the continuity equation \eqref{CE} satisfies
	\begin{equation}\label{eq:counterexample}
	\int_{B_{1/3}}\int_{\mathbb{R}^d}\frac{|u_t(x+h)-u_t(x)|^2}{|h|^d}\frac{1}{\log(1/|h|)^{\gamma}} \di x\di h=\infty,
	\qquad \forall\gamma<1-p,
	\end{equation}
for every $t>0$.
\end{theorem}

As a consequence of \autoref{prop: generalFact} (see also \autoref{remark: moreGeneral fact}) and \autoref{prop: regularity RLF} the result in \autoref{thm: sharpness} immediately implies the following.
\begin{proposition}
	Let $p\geq 1$ fixed. For every $q>p$ there exists a compact supported divergence-free vector field $b\in L^\infty([0,+\infty); W^{1,p}(\setR^d))$ whose regular Lagrangian flow $X$ satisfies the following property:
	for every $t>0$, $g\in L^q(\setR^d)$ and $\alpha\in (0,1]$ there exists a set $E\subset\setR^d$ of positive $\leb^d$ measure such that
	\begin{equation}
	|X_t(x)-X_t(y)|> |x-y|^{\alpha}\exp\left\lbrace g(x)+g(y)\right\rbrace\qquad \forall x,y\in E.
	\end{equation}
\end{proposition} 
In other words the exponential Lusin-Lipschitz regularity of order $p$ for Lagrangian flows associated to vector fields belonging to $W^{1,p}$ cannot be improved. Even an exponential Lusin H\"older regularity of order greater than $p$ cannot be reached. Compare this with \autoref{th: converse}.

The main idea behind our constructions comes from the work \cite{AlbertiCrippaMazzuccato16} by Alberti, Crippa and Mazzucato. In this paper the authors built a solution to \eqref{CE}, drifted by a divergence-free Sobolev vector field, that is smooth at time zero but it does not belong to any Sobolev space for positive times.
The construction of a vector field $b$ and the solution $u_t$ is achieved by patching together a countable number of pairs $v_n$ and $\rho_n$ of velocity fields and solutions to the Cauchy problem \eqref{CE} with disjoint supports. They are obtained by rescaling in space, time and size suitable building blocks provided by the following:

\begin{proposition}\label{alberti,Crippa,Mazzucato} Assume $d\geq 2$ and let $Q$ be the open cube with unit side centered at the origin of $\mathbb{R}^d$. There exist a velocity field $v\in C^\infty([0,+\infty)\times\mathbb{R}^d)$ and a (non trivial) solution $\rho\in  L^\infty([0,+\infty)\times\mathbb{R}^d) $ of the continuity equation \eqref{CE} such that 
	\begin{itemize}
		\item[(i)]$v_t$ is bounded, divergence-free and compactly supported in $Q$ for any $t\ge 0$;
		\item[(ii)] $\rho_t$ has zero average and it is bounded and compactly supported in $Q$ for any $t\ge 0$;
		\item[(iii)] $\sup_{t\ge 0}\norm{v_t}_{W^{1,p}(\setR^d)}<\infty$ for any $t\ge 0$, for any $1\le p\le \infty$;
		\item[(iv)] there exists a constant $c>0$ such that
		 \begin{equation}\label{z10}
		      \norm{\rho_t}_{\dot H^{-1}(\mathbb{R}^d)}\lesssim \exp(-ct),\qquad \forall t\ge 0,
		 \end{equation}
		 where $\norm{\cdot}_{\dot H^{-1}}$ is the negative homogeneous Sobolev norm of order $-1$.
	\end{itemize}	
\end{proposition}

The result in \autoref{alberti,Crippa,Mazzucato}, that is taken from \cite[Theorem 8]{AlbertiCrippaMazzuccato18}, provides a solution of \eqref{CE} (associated to a smooth divergence-free vector field) whose $\dot H^{-1}$ norm decays exponentially fast in time. 
It can be shown that, under the uniform $W^{1,p}$ bound with $p>1$ on the vector field, the rate of decay of negative Sobolev norms for solutions of \eqref{CE} is at most exponential (see \autoref{prop: mixing estimate} and the discussion above for more details), thus $\rho_t$ saturates this rate.

Observe that, by means of \eqref{z10}, \autoref{remark: norm is conserved} and the interpolation inequality
\begin{equation*}\label{z9}
\norm{\rho_t}_{L^2}^2\leq \norm{\rho_t}_{\dot H^{-1}}\norm{\rho_t}_{\dot H^1},
\end{equation*}
we deduce 
\begin{equation*}
	\norm{\rho_t}_{\dot H^1} \gtrsim \norm{\rho_0}_{L^1}\exp(ct), \qquad \forall t\ge 0,
\end{equation*}
namely, $\rho_t$ loses the Sobolev regularity exponentially fast in time.
This heuristic suggests to introduce a new interpolation inequality involving log-Sobolev functionals, negative Sobolev norms and $L^p$ norms in order to prove \autoref{thm:sharpness2}. It is the content of the next subsection.

%

\subsection{Interpolation inequality and proof of \autoref{thm:sharpness2}}\label{subsectio: Interpolation}
The main result of this subsection is inspired by \cite[Proof of Theorem 2.4]{Daodiaznguyen} and reads as follows.

\begin{proposition}\label{interpolation}
	Let us fix parameters $\gamma\in (-\infty,1)$, $\lambda \in (0,1/100)$ and $\delta\in (0,1]$. For any $f\in L^2(\setR^d)$ it holds
\begin{equation}\label{eq: interpolation}
    \norm{f}_{L^2}^2
    \lesssim_{d,\gamma}   \frac{1}{|\log(\delta\lambda)|^{1-\gamma}}\int_{B_{\frac{1}{5\delta}}}\int_{\mathbb{R}^d}\frac{|f(x+h)-f(x)|^2}{|h|^d|\log(\delta |h|)|^\gamma}\di x\di h+|\log(\lambda)|\frac{\norm{f}_{L^2}^2}{\log\left(2+\frac{\norm{f}_{L^2}}{\norm{f}_{\dot H^{-1}}}\right)}.
\end{equation}
\end{proposition}

Let us now describe two important corollaries of \autoref{interpolation}, that among other things, imply that $v$ and $\rho$ in \autoref{alberti,Crippa,Mazzucato} fulfill the assumptions of \autoref{thm:sharpness2}. 

The first corollary provides an estimate for a scaled version of the log-Sobolev functional \eqref{eq: log-Sobolev norm} applied to $\rho_t$ (the solution obtained in \autoref{alberti,Crippa,Mazzucato}) that plays a role in the proof of \autoref{thm: sharpness}. It can be proved using \eqref{eq: interpolation}, (iv)\autoref{alberti,Crippa,Mazzucato} and \autoref{remark: norm is conserved}.

\begin{corollary}\label{cor: log estimate}
	Let $\rho$ be as in \autoref{alberti,Crippa,Mazzucato} and $t>0$ be fixed. For every $\gamma\in (-\infty,1)$, $\lambda \in (0, 1/100)$ and $\gamma\in (0,1]$, it holds
	\begin{equation*}
	\int_{B_{\frac{1}{5\delta}}}\int_{\setR^d}\frac{|\rho_t(x+h)-\rho_t(x)|^2}{|h|^d|\log(\delta |h|)|^\gamma}\di x\di h
	\gtrsim_{d,\gamma}
	\norm{\rho_0}^2_{L^2} |\log(\delta \lambda)|^{1-\gamma}
	\left(
	C_{\gamma}-|\log(\lambda)|\frac{C(\norm{\rho_0}_{L^2})}{1+ct}
	\right),
	\end{equation*}
	where $c$ is the constant in \autoref{alberti,Crippa,Mazzucato}, $C_{\gamma}>0$ depends only on $\gamma$ and $C(\norm{\rho_0}_{L^2})>0$ depends only on $\norm{\rho_0}_{L^2}$.
\end{corollary}

The second corollary of \autoref{interpolation} follows from \eqref{eq: interpolation} setting $\delta=1$ and 
\begin{equation*}
|\log (\lambda)|=
\left[
\frac{\log\left(2+\frac{\norm{f}_{L^2}}{\norm{f}_{\dot H^{-1}}}\right)}{\norm{f}_{L^2}^2}
\int_{B_{1/5}}\int_{\mathbb{R}^d}\frac{|f(x+h)-f(x)|^2}{|h|^d\log(1/ |h|)^\gamma}\di x\di h\right]^{\frac{1}{2-\gamma}}.
\end{equation*}

\begin{corollary}\label{cor: interpolation}
	For every parameter $\gamma\in (-\infty,1)$ it holds
	\begin{equation}\label{z11}
\log\left(2+\frac{\norm{f}_{L^2}}{\norm{f}_{\dot H^{-1}}}\right)^{1-\gamma} \norm{f}_{L^2} 
	\lesssim_{d,\gamma} 	
	\int_{B_{1/5}}\int_{\mathbb{R}^d}\frac{|f(x+h)-f(x)|^2}{|h|^d}\frac{1}{\log(1/ |h|)^\gamma}\di x\di h,
	\end{equation}
	for every $f\in L^2(\setR^d)$.      
\end{corollary}
\autoref{cor: interpolation} immediately implies \autoref{thm:sharpness2}: apply \eqref{z11} with $\gamma=1-p$ to $\rho_t$ (the solution built in \autoref{alberti,Crippa,Mazzucato}) and use (iv)\autoref{alberti,Crippa,Mazzucato} and  \autoref{remark: norm is conserved}.

The remaining part of this section is devoted to the proof of \autoref{interpolation}.

\begin{proof}[Proof of \autoref{interpolation}]  
	 Fix $\varepsilon>0$. Let us fix $\phi\in C_c^\infty(\setR^d)$
	 such that $\phi=1$ in $B_1\setminus B_{1/2}$, $\phi=0$ in $\left(B_{5/4}\setminus B_{1/4}\right)^c$ and $\int_{\setR^d}\phi=1$. Set $\phi_\varepsilon(x):=\varepsilon^{-d}\phi(x/\varepsilon)$.
	 Thus we have 
	\begin{align*}
	\norm{f\ast \phi_\varepsilon}_{L^2}^2&=	\norm{\hat f\hat \phi_\varepsilon}_{L^2}^2
	\leq \norm{\log(2+|\cdot|)|\hat \phi_\varepsilon(\cdot)|^2}_{L^{\infty}}\int_{\setR^d}\frac{|\hat{f}(\xi)|^2}{\log(2+|\xi|)}\di \xi\\
	&=
	\norm{\log(2+\eps^{-1}|\cdot|)|\hat \phi(\cdot)|^2}_{L^{\infty}}\int_{\mathbb{R}^d}\frac{|\hat{f}(\xi)|^2}{\log(2+|\xi|)}\di \xi\\
	& \lesssim_d \abs{\log\left(\varepsilon\wedge \frac{1}{2}\right)} \int_{\mathbb{R}^d}\frac{|\hat{f}(\xi)|^2}{\log(2+|\xi|)}\di \xi,
	\end{align*}
	and
	\begin{align*}
	\norm{f-f\ast \phi_\varepsilon}_{L^2}^2\lesssim  \int_{\varepsilon\leq 4|h|\leq 5\varepsilon}\int_{\mathbb{R}^d}\frac{|f(x+h)-f(x)|^2}{|h|^d}\di x\di h.
	\end{align*}
	Therefore
	\begin{align}\label{z4}
	\norm{f}_{L^2}^2\lesssim_d \int_{\varepsilon\leq 4|h|\leq 5\varepsilon}\int_{\mathbb{R}^d}\frac{|f(x+h)-f(x)|^2}{|h|^d}\di x\di h+\abs{\log\left(\varepsilon\wedge \frac{1}{2}\right)} \int_{\mathbb{R}^d}\frac{|\hat{f}(\xi)|^2}{\log(2+|\xi|)}\di \xi.
	\end{align}
	Now we integrate \eqref{z4} with respect to a variable $\eps$ against a suitable kernel obtaining
	\begin{align}\nonumber
	\int_{\lambda}^{\frac{1}{10\delta}}\frac{1}{|\log(\delta\varepsilon)|^\gamma}\frac{\di \varepsilon}{\varepsilon}	\norm{f}_{L^2(\mathbb{R}^d)}^2&\lesssim_d  \int_{\lambda}^{\frac{1}{10\delta}} \int_{\varepsilon\leq 4|h|\leq 5\varepsilon}\int_{\mathbb{R}^d}\frac{|f(x+h)-f(x)|^2}{|h|^d}\di x\di h\frac{1}{|\log(\delta\varepsilon)|^\gamma}\frac{\di\varepsilon}{\varepsilon}	\\&+\int_{\lambda}^{\frac{1}{10\delta}}\frac{\abs{\log\left(\varepsilon\wedge \frac{1}{2}\right)}}{|\log(\delta\varepsilon)|^{\gamma}}\frac{\di \varepsilon}{\varepsilon} \int_{\mathbb{R}^d}\frac{|\hat{f}(\xi)|^2}{\log(2
		+|\xi|)}\di \xi.
	\label{es re-interpolation1}
	\end{align}
	Starting from the elementary inequalities
	\begin{equation*}
		-\frac{1}{1-\gamma}\frac{\di}{\di \varepsilon}|\log(\delta\varepsilon)|^{1-\gamma}=\frac{1}{|\log(\delta\varepsilon)|^\gamma}\frac{1}{\varepsilon} 
	\qquad \text{for}\ \varepsilon<\frac{1}{\delta},
	\end{equation*}
	and 
	\begin{equation*}
		-\frac{\di}{\di \varepsilon}\frac{|\log(\varepsilon)|^2}{|\log(\delta\varepsilon)|^{\gamma}}
		=\frac{|\log(\varepsilon)|}{|\log(\delta\varepsilon)|^{\gamma}}\frac{1}{\varepsilon}\left( 2-\gamma\frac{|\log(\varepsilon)|}{|\log(\delta\varepsilon)|}\right)
		\geq  \frac{|\log(\varepsilon)|}{|\log(\delta\varepsilon)|^{\gamma}}\frac{1}{\varepsilon} \qquad\text{for}\ \varepsilon < 1,
	\end{equation*} 
	we deduce
	\begin{equation}\label{z5}
	\int_{\lambda}^{\frac{1}{10\delta}}\frac{1}{|\log(\delta\varepsilon)|^\gamma}\frac{\di \varepsilon}{\varepsilon} \simeq_{\gamma}
	 |\log(\delta\lambda)|^{1-\gamma},
	\end{equation}
	when $\delta \lambda$ is small enough (for instance we can ask $\delta\lambda<1/100$, that is verified under our assumptions), and
	\begin{equation}\label{z6}
	 \int_{\lambda}^{\frac{1}{10\delta}}\frac{\abs{\log\left(\varepsilon\wedge \frac{1}{2}\right)}}{|\log(\delta\varepsilon)|^{\gamma}}\frac{\di \varepsilon}{\varepsilon}
	 \lesssim_{\gamma} |\log(\delta \lambda)|^{1-\gamma}\left( \frac{|\log(\lambda)^2}{\log(\delta \lambda)}+\left(  \frac{|\log(\delta)|}{|\log(\delta\lambda)|}\right)^{1-\gamma}\right)
	 \lesssim_{\gamma}|\log(\delta \lambda)|^{1-\gamma}|\log(\lambda)|,
	\end{equation}
	for every $\delta>0$. 
	Putting \eqref{es re-interpolation1}, \eqref{z5} and \eqref{z6} together we get
	\begin{equation*}
	\norm{f}_{L^2}^2
	\lesssim_{d,\gamma}   \frac{1}{|\log(\delta\lambda)|^{1-\gamma}}\int_{B_{\frac{1}{5\delta}}}\int_{\mathbb{R}^d}\frac{|f(x+h)-f(x)|^2}{|h|^d|\log(\delta |h|)|^\gamma}\di x\di h+|\log(\lambda)|\int_{\mathbb{R}^d}\frac{|\hat{f}(\xi)|^2}{\log(2+|\xi|)}\di \xi.
	\end{equation*}
	In order to conclude the proof it remains only to show	
	\begin{equation}\label{interpolation1}
		\int_{\mathbb{R}^d}\frac{|\hat{f}(\xi)|^2}{\log(2+|\xi|)}\di \xi\leq  \frac{2\norm{f}_{L^2}^2}{\log(2+\frac{\norm{f}_{L^2}}{\norm{f}_{\dot H^{-1}}})}.
	\end{equation}
	To this aim we fix a parameter $\nu>0$, and we estimate
		\begin{align*}
		\int_{\mathbb{R}^d}\frac{|\hat{f}(\xi)|^2}{\log(2+|\xi|)}\di \xi
		&=\int_{|\xi|\leq \nu}\frac{|\xi|^2}{\log(2+|\xi|)}|\xi|^{-2}|\hat{f}(\xi)|^2\di\xi+\int_{|\xi|\geq \nu}\frac{1}{\log(2+|\xi|)}|\hat{f}(\xi)|^2\di\xi\\
		&\leq \frac{\nu^2}{\log(2+\nu)}\int_{|\xi|\leq \nu}|\xi|^{-2}|\hat{f}(\xi)|^2\di\xi+\frac{1}{\log(2+\nu)}\int_{|\xi|\geq \nu}|\hat{f}(\xi)|^2\di\xi
		\\
		&\leq \frac{\nu^2}{\log(2+\nu)}\norm{f}_{\dot H^{-1}}^2+\frac{1}{\log(2+\nu)}\norm{f}_{L^2}^2.
		\end{align*}
		Choosing $\nu=\frac{\norm{f}_{L^2}}{\norm{f}_{\dot H^{-1}}}$, one gets \eqref{interpolation1}.		
\end{proof}

\subsection{Proof of \autoref{thm: sharpness}}
Before going into details with the proof of \autoref{thm: sharpness} we present the last technical ingredient: an approximate orthogonality property, with respect to the log-Sobolev functional \eqref{eq: log-Sobolev norm}, for functions with disjoint supports.

\begin{lemma} \label{subesti}
	Let $\gamma\in (-\infty,1)$ be fixed. For every $n\in \setN$ consider an open set $\Omega_n$, a function $f_n\in L^2(\setR^d)$ and a parameter $0<\lambda_n<1/4$. Assume that the family $\set{\Omega_n}_{n\in \setN}$ is disjoint and that the distance between $\supp f_n$ and $\setR^d\setminus \Omega_n$ is bigger than $\lambda_n$.
	
	Then it holds
	\begin{align}\nonumber
	\int_{B_{1/3}} & \int_{\mathbb{R}^d}\frac{|\sum_n f_n(x+h)-\sum_n f_n(x)|^2}{|h|^d}\frac{1}{\log(1/|h|)^\gamma} \di x\di h\\
	&\ge \limsup_{N\to\infty} \sum_{n=1}^N \left( \int_{B_{1/3}}\int_{\mathbb{R}^d}\frac{|f_n(x+h)-f_n(x)|^2}{|h|^d\log(1/|h|)^\gamma} \di x\di h-\frac{4\norm{f_n}_{L^2}^2}{1-\gamma}|\log(\lambda_n)|^{1-\gamma}\right).
	\label{essumf}
	\end{align}
\end{lemma}
\begin{proof} 
	Let us call $\bar \Omega_n\subset \Omega_n$ the set of $x\in \setR^d$ whose distance from $\supp f_n$ is smaller than $\lambda_n/3$. 
	Observe that
	\begin{align*} \int_{B_{1/3}} \int_{\mathbb{R}^d}&\frac{|\sum_n f_n(x+h)-\sum_n f_n(x)|^2}{|h|^d}\frac{1}{\log(1/|h|)^\gamma} \di x\di h\\
	\geq & \limsup_{N\to\infty} \sum_{n=1}^N \int_{B_{\lambda_n/3}}\int_{\bar \Omega_n}\frac{|f_n(x+h)-f_n(x)|^2}{|h|^d}\frac{1}{\log(1/|h|)^\gamma} \di x\di h\\
	=& \limsup_{N\to\infty} \sum_{n=1}^N\Big( \int_{B_{1/3}}\int_{\setR^d}\frac{|f_n(x+h)-f_n(x)|^2}{|h|^d}\frac{1}{\log(1/|h|)^\gamma}\di x\di h\\
	&\qquad\qquad\qquad-\int_{B_{1/3}\setminus B_{\lambda_n/3}}\int_{\mathbb{R}^d}\frac{|f_n(x+h)-f_n(x)|^2}{|h|^d}\frac{1}{\log(1/|h|)^\gamma} \di x\di h\Big).
	\end{align*}
	On the other hand, 
	\begin{align*}
	&	\int_{B_{1/3}\setminus B_{\lambda_n/3}}\int_{\mathbb{R}^d}\frac{|f_n(x+h)-f_n(x)|^2}{|h|^d}\frac{1}{\log(1/|h|)^\gamma} \di x\di h\\
	&\leq 2\int_{B_{1/3}\setminus B_{\lambda_n/3}}\int_{\setR^d}\frac{|f_n(x)|^2}{|h|^d\log(1/|h|)^\gamma}\di x\di h+2\int_{B_{1/3}\setminus B_{\lambda_n/3}}\int_{\setR^d}\frac{|f_n(x+h)|^2}{|h|^d\log(1/|h|)^\gamma} \di x\di h
	\\
	&\le 4\norm{f_n}_{L^2}^2\int_{B_{1/3}\setminus B_{\lambda_n/3}}\frac{1}{|h|^d} \frac{1}{\log(1/|h|)^\gamma}\di h \le \frac{4\norm{f_n}_{L^2}^2}{1-\gamma}|\log(\lambda_n)|^{1-\gamma}.
	\end{align*}
	This concludes the proof.
\end{proof}

\begin{proof}[Proof of \autoref{thm: sharpness}]
	Let $p\ge 1$ be fixed.
	We consider $v$ and $\rho$ as in \autoref{alberti,Crippa,Mazzucato}, and a family of disjoint open cubes $\set{Q_n}_{n\in \setN}$ contained in $B_1$. Let us define
	\begin{equation}\label{eq: parameters}
	\lambda_n:=e^{-n},\qquad\gamma_n:=\frac{1}{n^2},\qquad \tau_n:=\left(n^2e^{-dn}\right)^{1/p},
	\qquad\text{for $n\in\setN$.}
	\end{equation}
	Assuming that $Q_n$ has side of length $3\lambda_n$ and center at $x_n\in B_1$, we set 
	\begin{align*}
	v_n(t,x):=\frac{\lambda_n}{\tau_n}v\left(\frac{t}{\tau_n},\frac{x-x_n}{\lambda_n}\right),
	\quad
	\rho_n(t,x):=\gamma_n\rho\left(\frac{t}{\tau_n},\frac{x-x_n}{\lambda_n}\right),
	\quad \text{for $x\in \setR^d$, $t\ge 0$, $n\in \setN$};
	\end{align*}
    observe that $u_n$ is supported in $Q_n$ and $\text{dist}(\supp u_n, \setR^d\setminus Q_n)\geq \lambda_n$ for every $n\in \setN$.
	
	Setting
	\begin{align*}
	b(t,x):=\sum_n v_n(t,x),\qquad u(t,x):=\sum_n \rho_n(t,x)
	\qquad \forall x\in \setR^d,\ \quad\forall t>0,
	\end{align*}
    the following facts hold true:
    \begin{itemize}
    	\item[(i)] $b$ is divergence-free, supported in $B_1\times [0,+\infty)$ and belongs to $L^{\infty}([0,+\infty); W^{1,p})$;
    	\item[(ii)] $u$ is bounded and supported in $B_1\times [0,+\infty)$;
    	\item[(iii)] the initial data $\bar u$ belongs to $W^{1,d}$;
    	\item[(iv)] $u$ is a solution of \eqref{CE} with vector field $b$.
   \end{itemize}
   Let us prove (i), (ii), (iii) and (iv). For any $t>0$ we have
   \begin{align*}
	\norm{b_t}_{W^{1,p}}^p &
	\leq \sum_n \left(\norm{v_n(t,\cdot)}^p_{L^p}+\norm{\nabla v_n(t, \cdot)}^p_{L^p}\right)\\
	&= \sum_n \left(\frac{\lambda_n}{\tau_n}\right)^p  \left( \lambda_n^{d}\norm{v_{t/\tau_n}}_{L^p}^p+ \lambda_n^{d-p}\norm{\nabla v_{t/\tau_n}}_{L^p}^p\right)\\
	&\leq  \sup_{s\ge 0}\norm{v_s}_{W^{1,p}}^p \sum_n \frac{\lambda_n^{d}}{\tau_n^p}
	=\sup_{s\ge 0}\norm{v_s}_{W^{1,p}}^p \sum_n \frac{1}{n^2}<\infty,
   \end{align*}
   this gives the non trivial part of (i). Point (ii) follows observing that $\sup_n \gamma_n<\infty$. In order to prove (iii) we estimate
\begin{align*}
	\norm{u_0}_{W^{1,d}}^d \leq  &  \sum_n \left( \norm{\rho_n(0, \cdot)}_{L^d}^d+\norm{\nabla \rho_n(0, \cdot)}_{L^d}^d\right)\\
	=& \sum_{n} \gamma_n^d \left( \lambda_n^d \norm{\rho_0}_{L^d}^d+ \norm{\nabla \rho_0}_{L^d}^d\right)\\
	\leq & \norm{\rho_0}_{W^{1,d}}^d \sum_n \gamma_n^d<\infty.
\end{align*}
Point (iv) is an immediate consequence of the construction.
	
We are now ready to show \eqref{eq:counterexample}. Fix $t>0$ and $\gamma\in (-\infty,1)$.
   Thanks to \autoref{subesti} and \autoref{remark: norm is conserved} we have
   	\begin{align*}
		& \int_{B_{1/3}}\int_{\mathbb{R}^d}\frac{|u(t,x+h)-u(t,x)|^2}{|h|^d\log(1/|h|)^\gamma} \di x\di h\\
		&\geq \limsup_{N\to\infty}\sum_{n=1}^N \left( \int_{B_{1/3}}\int_{\mathbb{R}^d}\frac{|\rho_n(t,x+h)-\rho_n(t,x)|^2}{|h|^d\log(1/|h|)^\gamma}\di x\di h-\frac{4\norm{\rho_n(t,\cdot)}_{L^2}^2}{1-\gamma}|\log(\lambda_n)|^{1-\gamma}\right)\\
		&= \limsup_{N\to\infty}\sum_{n=1}^N \gamma_n^2 \left( \int_{B_{1/3}}\int_{\mathbb{R}^d}\frac{|\rho\left(\frac{t}{\tau_n},\frac{x+h-x_n}{\lambda_n}\right)-\rho\left(\frac{t}{\tau_n},\frac{x-x_n}{\lambda_n}\right)|^2}{|h|^d\log(1/|h|)^\gamma} \di x\di h-\frac{4\norm{\rho\left(\frac{t}{\tau_n},\frac{\cdot}{\lambda_n}\right)}_{L^2}^2}{1-\gamma}|\log(\lambda_n)|^{1-\gamma}\right)
		\\
		&= \limsup_{N\to\infty}\sum_{n=1}^N \gamma_n^2\lambda_n^d \left( \int_{B_{\frac{1}{3\lambda_n}}}\int_{\mathbb{R}^d}\frac{|\rho\left(\frac{t}{\tau_n},x+h\right)-\rho\left(\frac{t}{\tau_n},x\right)|^2}{|h|^d|\log(|\lambda_n h|)|^\gamma} \di x\di h-\frac{4\norm{\rho_0}_{L^2}^2}{1-\gamma}|\log(\lambda_n)|^{1-\gamma}\right).
	\end{align*}
		Let us fix $n\in \setN$ and $\lambda\in (0,1/100)$. Applying \autoref{cor: log estimate} with parameters $\gamma$, $\lambda$, and $\delta=\lambda_n$ (we need to consider $n$ bigger than a suitable integer $n_{\gamma}$ depending only on $\gamma$) we get
         \begin{align*}	     \int_{B_{\frac{1}{3\lambda_n}}}&\int_{\mathbb{R}^d}\frac{|\rho\left(\frac{t}{\tau_n},x+h\right)-\rho\left(\frac{t}{\tau_n},x\right)|^2}{|h|^d|\log(|\lambda_n h|)|^\gamma} \di x\di h \\
            & \gtrsim_{\gamma} \norm{\rho_0}^2_{L^2}|\log(\lambda_n\lambda)|^{1-\gamma}\left(
            C_{\gamma}-|\log(\lambda)|\frac{\tau_nC(\norm{\rho_0}_{L^2})}{\tau_n+ct}
            \right).
		\end{align*}
		For $n\in\setN$ big enough we take $\lambda$ such that, 
		\begin{equation*}
		    a_t\tau_n^{-1}= |\log(\lambda)|
		    \quad\text{where}
		    \quad
			a_t:=\frac{C_{\gamma}}{2C(\norm{\rho_0}_{L^2})}ct,
		\end{equation*}
		obtaining
		\begin{align*}
		|\log(\lambda_n\lambda)|^{1-\gamma}\left(
		C_{\gamma}-|\log(\lambda)|\frac{\tau_nC(\norm{\rho_0}_{L^2})}{\tau_n+ct}\right) \geq 
		(|\log(\lambda_n)|+a_t(\tau_n)^{-1})^{1-\gamma}\frac{C_{\gamma}}{2}
		\geq \bar{C}  t^{1-\gamma}\tau_n^{\gamma-1},
		\end{align*}
		for every $n\geq n_{\gamma}$,
		where $\bar{C}$ is a positive constant depending only on, $c$ (see \autoref{alberti,Crippa,Mazzucato}), $\norm{\rho_0}_{L^2}$ and $\gamma$.
		
		Putting all together and recalling \eqref{eq: parameters} we get
		\begin{align*}
		 \int_{B_{1/3}}\int_{\mathbb{R}^d} &\frac{|u(t,x+h)-u(t,x)|^2}{|h|^d}\frac{1}{\log(1/|h|)^\gamma} \di x\di h\\
		 \geq & \bar{C}t^{1-\gamma}\sum_{n=n_{\gamma}}^{\infty} \gamma_n^2 \lambda_n^d \tau_n^{\gamma-1}
		 -\frac{4\norm{\rho_0}_{L^2}}{1-\gamma}\sum_{n=1}^{\infty}\gamma_n^2\lambda_n^d|\log(\lambda_n)|^{1-\gamma}\\
		 = & \bar{C}t^{1-\gamma}\sum_{n={n_{\gamma}}}^{\infty}  n^{\frac{2(\gamma-1)}{p}-4}  e^{-dn\frac{\gamma+p-1}{p}}-\frac{4\norm{\rho_0}_{L^2}}{1-\gamma}\sum_{n=1}^{\infty}n^{-\gamma-3} e^{-dn},
		\end{align*}
        that is equal to $+\infty$ when $\gamma<1-p$ and $t>0$.
\end{proof}

\subsection{Mixing estimates}
As a simple byproduct of our results (\autoref{thm:regularity result}, \autoref{cor: interpolation}) we get two mixing estimates for solutions of \eqref{CE} drifted by divergence-free vector fields bounded in $W^{1,p}$, uniformly in time, for $p>1$. 

These results are already present in the literature (see \cite[Theorem 6.2]{CrippaDeLellis08}, \cite{IyerKiselevXu14}, \cite{Seis13}, \cite{LegerFlavien16}), it is worth mentioning that the extension to the case $p=1$ is an important open problem related to the Bressan's mixing conjecture (see \cite{Bressan03}).

Let us begin with a simple estimate involving the geometric mixing scale. 

\begin{lemma}\label{lem-mix-ine} 
	Let $p>0$, $f\in L^\infty(\mathbb{R}^d)$ be such that $\norm{f}_{L^\infty}=1$ and $\leb^d(\{|f|=1\})\geq c_0>0$. Then, for any $\kappa\in (0,1)$ and $\varepsilon\in (0,1/3)$, it holds
	\begin{equation}\label{mix-ine}
	\sup_{x\in \setR^d}\left|\dashint_{B_\eps}f(x+y)\di y\right|<\kappa \implies
     \eps\geq \exp\left\lbrace -C  \left(\sup_{h\in B_{1/3}}\log(1/|h|)^p\int_{\setR^d}|f(x+h)-f(x)|\di x\right)^{1/p}\right\rbrace ,
	\end{equation}
	where $C=(c_0(1-\kappa))^{-1/p}.$
\end{lemma}
\begin{proof} For any  $x\in \left\{|f|=1\right\}$, we have $
	1-\kappa<\left|\dashint_{B_\varepsilon}(f(x)-f(x+h))\di h \right|$. 
	Therefore
	\begin{align*}
	c_0(1-\kappa)&< \dashint_{B_\varepsilon}\int_{\mathbb{R}^d}|f(x)-f(x+h)|\di x\di h\\
	&\leq \left(\sup_{h\in B_{1/3}}\log(1/|h|)^p\int_{\mathbb{R}^d}|f(x)-f(x+h)|dx\right) \dashint_{B_\varepsilon}\log(1/|h|)^{-p}\di h
	\\
	&\leq \left(\sup_{h\in B_{1/3}}\log(1/|h|)^p\int_{\mathbb{R}^d}|f(x)-f(x+h)|\di x\right) |\log(\varepsilon)|^{-p},
	\end{align*}
	this implies \eqref{mix-ine}. The proof is complete.
	
\end{proof}
We are now ready to state and prove the aforementioned mixing estimates.

\begin{proposition}\label{prop: mixing estimate}
	Let $p>1$ be fixed.
	Let us consider a bounded divergence-free vector field $b$ such that
	\begin{equation*}
	\norm{\nabla b_t}_{L^p}\leq B<\infty 
	\qquad
	\text{for a.e.}\ t\geq 0.
	\end{equation*}
	Then for every initial data $\bar u\in BV(\setR^d)\cap L^{\infty}(\setR^d)$ the (unique) solution $u\in L^{\infty}([0,+\infty)\times \setR^d)$ of \eqref{CE} satisfies
	\begin{equation}\label{eq: mixing estimate}
	\norm{u_t}_{\dot H^{-1}} \geq C\exp(-cBt),
	\qquad \text{for any $t\geq 0$},
	\end{equation}
	where $C>0$ and $c>0$ depend only on $\norm{\bar u}_{L^2}$, $\norm{\bar u}_{BV}$, $p$ and $d$. 
	
	Furthermore, if we assume $u_0(x)\in \set{1,-1,\ 0}$ for every $x\in \setR^d$ and $\int_{\setR^d} |u_0| \di x\geq c_0>0$, then for any $\kappa\in (0,1)$, and $\varepsilon\in (0,1/3)$ it holds
	\begin{equation}\label{eq: geometric mixing}
	\sup_{x\in \mathbb{R}^d}\left|\dashint_{B_\varepsilon(x)}u_t(y)\di y\right|<\kappa \Longrightarrow  	\eps\geq  \exp(-CBt),
	\end{equation}
	where $C>0$  depends only on $p$, $d$, $\kappa$, $c_0$ and $\norm{u_0}_{BV}$.
\end{proposition}
\begin{proof}
	The first part of the statement follows applying \autoref{cor: interpolation} with $\gamma=1-p$ and \autoref{thm:regularity result}.
	The second part is a consequence of \autoref{lem-mix-ine}, \autoref{remark: sup-norm} and the following elementary observation:
	if $f$ is a measurable function that takes only the values $1$, $0$ and $-1$ then
	\begin{equation*}
		\int_{\setR^d}|f(x+h)-f(x)| \di x \leq 	\int_{\setR^d}|f(x+h)-f(x)|^2 \di x
		\qquad
		\text{for every $h\in\setR^d$}.
	\end{equation*}
\end{proof}
Two remarks are in order.
\begin{remark}
	 A version of the mixing estimate \eqref{eq: mixing estimate} holds true also using the seminorm $\dot H^{-s}$ with $s>0$ and reads 
	\begin{equation*}
	\norm{u_t}_{\dot H^{-s}} \geq C\exp(-cBst)
	\qquad \text{for any $t\geq 0$}.
	\end{equation*}
	It can be proved using a modified version of the interpolation inequality \eqref{eq: interpolation}.
\end{remark}
\begin{remark}
	Let us assume $b_t$ to be smooth and compactly supported in $Q=[0,1]^d\subset\setR^d$ and let us call $X_t$ its flow.
	Let us consider the initial data $\bar u:=\mathbf{1}_A-\mathbf{1}_{Q\setminus A}\in BV(\setR^d)$ with $A\subset [0,1]^d$ satisfying $\leb^d(A)=\frac{1}{2}$. 
	It is immediate to see that $u_t:=\mathbf{1}_{A_t}-\mathbf{1}_{Q\setminus A_t}$ where $A_t=X_t(A)$. Setting $\kappa=1/2$, fixing $p>1$ and using \eqref{eq: geometric mixing} we get
	\begin{equation}\label{eq: Bressan p grater than 1}
	\frac{1}{4}\leq  \frac{\leb^d(B_{\eps}(x)\cap A_t) }{\omega_d \eps^d}\leq \frac{3}{4}
	\implies
	\int_0^t \norm{\nabla b_s}_{L^p} \di s \geq C |\log(\eps)|,	
	\end{equation}
	where $C$ depends on $p,d$ and $\norm{u_0}_{BV}$.
	
	\eqref{eq: Bressan p grater than 1} is the statement of Bressan's conjecture for $p>1$ (see \cite{Bressan03}) that has been proved for a first time in \cite{CrippaDeLellis08}.
\end{remark}

    Let us conclude the paper with an open question.
\begin{openquestion}\label{question}
	Let $b\in L^\infty([0,+\infty); W^{1,1}(\setR^d,\setR^d))$ be a divergence-free vector field with compact support.  Fix an initial data $\bar{u}\in C^\infty_c(\setR^d)$ and consider $u_t$ the unique solution in $L^{\infty}([0,+\infty)\times \setR^d)$ of \eqref{CE}.
	Is there an increasing function $\psi :(0,+\infty)\to (0,+\infty)$ (possibly depending on $b$) with $\lim_{s\to\infty} \psi(s)=\infty$ and $ \psi^{-1}(2s)\leq C\psi^{-1}(s)$ such that 
	\begin{equation*}
	\sup_{h\in B_{1/3}} \psi\left(\log(1/|h|)\right)\int_{\setR^d}|u_t(x+h)-u_t(x)|^2\di x\leq C\psi\left(t\right)<\infty,
	\end{equation*}
	for every $t$ big enough?
	\end{openquestion}

  A positive answer of \autoref{question}, together with the proof of  \autoref{lem-mix-ine}, gives an exponential bound on mixing in the case $p=1$, and thus solves the full mixing conjecture proposed by Bressan (see \cite{Bressan03}).

\end{document}